\documentclass[12pt,reqno]{amsart}

\usepackage{amsmath,amsfonts,amsthm,amssymb}

\newtheorem{theorem}{Theorem}[section]

\newtheorem{prop}[theorem]{Proposition}
\newtheorem{conj}[theorem]{Conjecture}

\theoremstyle{definition}
\newtheorem{definition}[theorem]{Definition}

\theoremstyle{remark}

\usepackage[a4paper,height=20cm,top=24mm,bottom=24mm,left=26mm,right=26mm]{geometry}
\usepackage[numbers,sort]{natbib}
\usepackage[breaklinks=true]{hyperref}
\usepackage{graphicx}
\usepackage{epstopdf}
\usepackage{booktabs}
\usepackage[table]{xcolor}
                                        
\DeclareMathAlphabet{\mathpzc}{OT1}{pzc}{m}{it}
\usepackage{titlesec}
\titleformat{\section}
{\bfseries\scshape\centering}
{\thesection.}{.5em}{}
\titleformat{\subsection}
{\rmfamily\bfseries}
{\thesubsection}{.5em}{}
\titleformat{\subsubsection}
{\rmfamily\bfseries}
{\thesubsubsection}{.5em}{}
\numberwithin{equation}{section}
\newcommand{\I}{\mathrm{i}}
\newcommand{\E}{\mathrm{e}}

\DeclareMathOperator{\CS}{CS}
\DeclareMathOperator{\Li}{Li_2}

\DeclareMathOperator{\sign}{sgn}

\DeclareMathOperator{\Vol}{Vol}

\DeclareMathDelimiter{\Norm}{\mathord}{largesymbols}{"3E}{largesymbols}{"3E}

\DeclareMathOperator{\Rogers}{L}
\begin{document}
%\allowdisplaybreaks[0]
\allowdisplaybreaks[4]
\abovedisplayskip=2pt
\belowdisplayskip=2pt
\baselineskip 16pt
\parskip 8pt
\sloppy

%%%%%%%%%%%%%%%%%% TITLE %%%%%%%%%%%%%%

\title[Braids,  Complex Volume \& Cluster Algebra]{Braids,  Complex Volume, and Cluster Algebra}

%%%%%%%%%%%%%%%%%%%%%%% AUTHOR(S) %%%%%%%%%%%%%%%%%%%

\author[K. Hikami]{Kazuhiro \textsc{Hikami}}

\address{Faculty of Mathematics,
  Kyushu University,
  Fukuoka 819-0395, Japan.}

\email{KHikami@gmail.com}

 \author[R. Inoue]{Rei \textsc{Inoue}}

 \address{Department of Mathematics and Informatics,
   Faculty of Science,
   Chiba University,
   Chiba 263-8522, Japan.}

 \email{reiiy@math.s.chiba-u.ac.jp}

%%%%%%%%%%%%%%%%%%%%%% DATE %%%%%%%%%%%%%%%%%%%%%%%%%%%
%(Received: \hspace{40mm})

%\vspace{18pt}
%\date{\today}
%\date{April 17, 2013, revised on February 25, 2014.}
\date{April 17, 2013, revised on November 1, 2014.}

%%%%%%%%%%%%%%%%%%%%%% ABSTRACT %%%%%%%%%%%%%%%%%%%%%%
\begin{abstract}
  We try to give a cluster algebraic interpretation of complex volume of
  knots.
  We construct
  the $\mathsf{R}$-operator
  from  the cluster
  mutations,
  and we show that it is regarded as a hyperbolic octahedron.
  The cluster variables are interpreted as edge parameters used by
  Zickert in computing complex volume.
\end{abstract}

%%%%%%%%%%%%%%%%%%%%%%% Key Words %%%%%%%%%%%%%%%%%%%%%%%%%%


%\subjclass[2000]{
%}

%%%%%%%%%%%%%%%%%%%%%%%%%%%%%%%%%%%%%%%%%%%%%%%%%%%%%%%%%
%\newpage

%%\renewcommand{\thefootnote}{\arabic{footnote}}
\maketitle

%%%%%
\section{Introduction}

Geometrical property of quantum invariants receives renewed interests
since proposed is the volume conjecture~\cite{Kasha96b,MuraMura99a},
which suggests a relationship between
the colored Jones polynomial and the hyperbolic volume of knot
complements.
As
quantum invariants of knots such as the colored Jones polynomial are
constructed  by use of the Artin braid relation,
it is interesting to 
study a hyperbolic geometrical solution of the braid relation.

A fundamental object in the 3-dimensional  hyperbolic geometry
is an ideal tetrahedron.
When
a manifold  is constructed from a set of ideal tetrahedra,
its
complex volume,
\emph{i.e.}, a
complexification of hyperbolic volume,
is written in terms of the extended
Rogers dilogarithm function~\cite{WDNeum00a}.
On the other hand,
the  cluster algebra
has been developed recently since
the pioneering  work~\cite{FominZelev02a}, and
there exist many applications including representation theory,
Teichm\"uller theory,  
integrable systems and so on.
The dilogarithm function plays an important
role also in the cluster algebra~\cite{FockGonc03,Nakanishi10}.

Purpose of this article is to give a cluster algebraic interpretation
of complex volume of knots.
In our previous paper~\cite{HikamiRInoue12a}, we gave an
interpretation of the cluster mutation as a hyperbolic ideal tetrahedron,
and we proposed a method to
compute the complex volume of 2-bridge knots.
In this article,
we first give a geometric interpretation of
the $\mathsf{R}$-operator, which can be constructed from
the cluster mutation based on a relation with the Teichm{\"u}ller
theory~\cite{Dynnik02a,DehDynRolWie08}.
We find that
the $\mathsf{R}$-operator in Theorem~\ref{thm:main}
is identified with a 
hyperbolic octahedron composed of four ideal tetrahedra, and
that the cluster variable corresponds to an edge parameter used by
Zickert~\cite{Zicke09}
for a computation of complex volume.
Our main claim is in Theorem~\ref{thm:main2} and Conjecture~\ref{conj}:
following a method of Zickert, we propose
a formula of complex volume in terms of cluster variables.
%as  Conjecture~\ref{conj}.
Our construction can be naturally quantized
with a help of the quantum cluster algebra~\cite{HikamiRInoue}.
% We note that a quantization of $\mathsf{R}$-operator is 
% constructed by using quantum cluster algebra \cite{HikamiRInoue}.
% It may link to a new aspect of the volume conjecture.

This paper is organized as follows.
In section 2, after explaining minimal basics of cluster algebra, 
we introduce the $\mathsf{R}$-operator.
In Section 3, we interpret the $\mathsf{R}$-operator in 
hyperbolic geometry, and formulate the complex volume of knots at  
Theorem~\ref{thm:main2} and Conjecture~\ref{conj}.
Some examples of numerical calculation
are presented in section 4.

%%%
\section{Cluster Algebra and Braid Relation}
\subsection{Cluster Variable}

We briefly introduce a notion of cluster algebra used in this article.
A basic reference is~\cite{FominZelev02a}.

A cluster seed $(\boldsymbol{x}, \mathbf{B})$
is a pair of
\begin{itemize}
\item a cluster variable $\boldsymbol{x}=(x_1, \dots, x_N)$: an
  $N$-tuple of algebraically
  independent variables,

\item an exchange matrix $\mathbf{B}=(b_{ij})$: an $N\times N$ skew
  symmetric integer matrix.

\end{itemize}
For each $k=1,\dots, N$, we define the mutation $\mu_k$ of
$(\boldsymbol{x},
% \boldsymbol{y},
\mathbf{B})$
by
\begin{equation}
  \mu_k
  \left(
    \boldsymbol{x},
% \boldsymbol{y} ,
    \mathbf{B}
  \right)
  =
  ( 
    \widetilde{\boldsymbol{x}},
%    \widetilde{\boldsymbol{y}},
    \widetilde{\mathbf{B}}
    ) ,
\end{equation}
where
\begin{itemize}
\item 
%a cluster variable 
  $\widetilde{\boldsymbol{x}} = 
  (\widetilde{x}_1, \dots, \widetilde{x}_N)$ is
  \begin{equation}
    \widetilde{x}_i
    =
    \begin{cases}
      x_i, 
      &
      \text{for $i \neq k$,}
      \\[2ex]
      \displaystyle
      \frac{1}{x_k} \,
      \left(
        \prod_{j: b_{jk}>0}
        x_j^{~b_{jk}}
        +
        \prod_{j: b_{jk} < 0}
        x_j^{~-b_{jk}}
      \right) ,
      &
      \text{for $i = k$,}
    \end{cases}
  \end{equation}

\item 
% an exchange matrix
  $\widetilde{\mathbf{B}}=
  ( \widetilde{b}_{ij} )$ is
  \begin{equation}
    \label{mutation_B}
    \widetilde{b}_{ij}
    =
    \begin{cases}
      - b_{ij},
      & \text{for $i=k$ or $j=k$},
      \\[2ex]
      \displaystyle
      b_{ij}
      + \frac{
        \bigl| b_{ik} \bigr| \, b_{kj}
        +
        b_{ik} \, \bigl| b_{kj} \bigr|
      }{2},
      & \text{otherwise.}
    \end{cases}
  \end{equation}
\end{itemize}
The pair $(\widetilde{\boldsymbol{x}}, \widetilde{\mathbf{B}})$
is again a cluster seed.
We remark that the mutation $\mu_k$ is involutive, 
and that
we have 
$\mu_j \,\mu_k (\boldsymbol{x}, \mathbf{B})
= \mu_k \, \mu_j (\boldsymbol{x}, \mathbf{B})$
if $b_{jk} = 0$.  

In terms of the cluster variable  $\boldsymbol{x}$, we introduce the
$y$-variable,
$\boldsymbol{y}=(y_1, \dots, y_N)$,
defined by~\cite{FominZelev07a}
\begin{equation}
  \label{y_from_x}
  y_j = \prod_k x_k^{~ b_{kj}} .
\end{equation}
The mutation $\mu_k$ induces a mutation of a pair
$( \boldsymbol{y}, \mathbf{B})$:
\begin{equation}
  \mu_k( \boldsymbol{y}, \mathbf{B})
  =
  (
    \widetilde{\boldsymbol{y}},
    \widetilde{\mathbf{B}}
    ) ,
\end{equation}
where $\widetilde{\mathbf{B}}$ is~\eqref{mutation_B}, and
$\widetilde{\boldsymbol{y}}=
(\widetilde{y}_1, \dots, \widetilde{y}_N)$
with
$\widetilde{y}_j = \prod_k \widetilde{x}_k^{~\widetilde{b}_{kj}}$
is given by
\begin{equation}
  \widetilde{y}_i
  =
  \begin{cases}
    y_k^{~-1} ,
    & \text{for $i = k$,}
    \\[2ex]
    y_i \, \left( 1+y_k^{~-1} \right)^{-b_{ki}} ,
    & \text{for $i \neq k$, $b_{ki} \geq 0$,}
    \\[2ex]
    y_i \left( 1+ y_k \right)^{-b_{ki}} ,
    & \text{for $i \neq k$, $b_{ki} \leq 0$.}
  \end{cases}
\end{equation}

%%%%
\subsection{\mathversion{bold}$\mathsf{R}$-operator}
We set the $7$ by $7$ exchange matrix $\mathbf{B}$ as
\begin{equation}
  \label{B_matrix_7}
  \mathbf{B} =
  \begin{pmatrix}
    0& 1 & -1 & 0 & 0 & 0 & 0     \\
    -1 & 0 & 0 & 1 & 0 & 0 & 0 \\
    1 & 0 & 0 & -1 & 0 & 0 & 0 \\
    0 & -1 & 1 & 0 & 1 & -1 & 0\\
    0 & 0 & 0 & -1 & 0 & 0 & 1 \\
    0 & 0 & 0 & 1 & 0 & 0 & -1 \\
    0 & 0 & 0 & 0 & -1 & 1 & 0
  \end{pmatrix} .
\end{equation}
By regarding the matrix element as
\begin{equation}
  \label{b_and_arrow}
  b_{ij}
  =
  \# \left\{ \text{arrows from $i$ to $j$} \right\}
  -
  \# \left\{ \text{arrows from $j$ to $i$} \right\} ,
\end{equation}
exchange matrix $\mathbf{B}$ corresponds to quiver, which is dual to
triangulated surface
(see, \emph{e.g.},~\cite{FomiShapThur08a}).
In our case~\eqref{B_matrix_7},
we have the quiver and
the  triangulated disk
depicted  in Fig.~\ref{fig:quiverR}.

\begin{figure}[tbhp]
  \begin{minipage}{0.47\linewidth}
    \centering
    \includegraphics{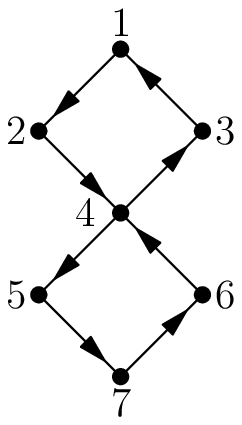}
  \end{minipage}
  \hfill
  \begin{minipage}{0.47\linewidth}
    \centering
    \includegraphics{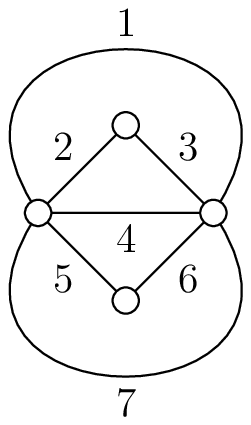}
  \end{minipage}
  \caption{Quiver and triangulated disk}
  \label{fig:quiverR}
\end{figure}

For our later use, we introduce the $\mathsf{R}$-operator acting on
the cluster variables associated with the quiver in
Fig.~\ref{fig:quiverR}.

\begin{definition}
  We define the $\mathsf{R}$-operator by
  \begin{equation}
    \label{R}
    \mathsf{R}=
    s_{3,5} \, s_{2,5} \,  s_{3,6}  \, \mu_4 \, \mu_2 \, \mu_6 \,
    \mu_4 .
  \end{equation}
\end{definition}

Here we have used the permutation $s_{i,j}$ of subscripts $i$ and $j$
in seeds,
\begin{equation*}
  s_{i,j} ( \dots, x_i, \dots, x_j , \dots)
  =
  (\dots, x_j, \dots, x_i, \dots) .
\end{equation*}
Actions on the exchange matrix are defined in the same manner.
Note that we have
%due to that the mutation $\mu_k$ is involutive, 
\begin{equation}
  \label{Rinv}
  \mathsf{R}^{-1} =
  s_{3,6} \, s_{2,5} \, s_{3,5} \, \mu_4 \, \mu_5 \, \mu_3 \, \mu_4 .
\end{equation}

The permutations are included in the $\mathsf{R}$-operator so that
the exchange matrix $\mathbf{B}$~\eqref{B_matrix_7}
is invariant under $\mathsf{R}$.
Explicitly  we have
\begin{align}
   \mathsf{R}^{\pm 1} (\boldsymbol{x}, \mathbf{B}) =
  (\mathsf{R}^{\pm 1}(\boldsymbol{x}),  \mathbf{B} ) ,
  % & \mathsf{R}(\boldsymbol{x}, \mathbf{B}) =
  % (\mathsf{R}(\boldsymbol{x}),  \mathbf{B} ) ,
  % &
  % & \mathsf{R}^{-1}(\boldsymbol{x}, \mathbf{B}) =
  % (\mathsf{R}^{-1}(\boldsymbol{x}),  \mathbf{B} ) ,
\end{align}
where
%\begin{equation}
%  \label{R_on_x}
  \begin{align}
%    &
%     \begin{aligned}
%     R(\boldsymbol{x})
%     & =
%     \Bigl(
%     x_1
%     ,
%     x_5
%     ,
%     \displaystyle
%     \frac{x_1 \, x_3 \, x_5 + x_3 \, x_4 \, x_5
%       + x_1 \, x_2 \,  x_6}{
%       x_2 \, x_4}
%     ,
%     \\
%     & \qquad
%     \displaystyle
%     \frac{
%       x_1 \, x_3 \, x_4 \, x_5+ x_3 \, x_4^{~2} \, x_5 +
%       x_1 \ x_3 \, x_5 \, x_7 + x_3 \, x_4 \, x_5 \, x_7
%       + x_1 \, x_2 \, x_6 \, x_7}{
%       x_2 \, x_4 \, x_6}
%     ,
%     \displaystyle
%     \frac{
%       x_3 \, x_4 \, x_5 + x_3 \, x_5 \, x_7 + x_2 \, x_6 \, x_7}{
%       x_4 \, x_6}
% ,
%     x_3 
% ,
%     x_7
%     \Bigr),
%   \end{aligned}
  %%% 
    &
    \mathsf{R}(\boldsymbol{x}) =
     \begin{pmatrix}
     x_1
     \\
     x_5
     \\[1ex]
     \displaystyle
     \frac{x_1 \, x_3 \, x_5 + x_3 \, x_4 \, x_5
       + x_1 \, x_2 \,  x_6}{
       x_2 \, x_4}
     \\[2ex]
     \displaystyle
     \frac{
       x_1 \, x_3 \, x_4 \, x_5+ x_3 \, x_4^{~2} \, x_5 +
       x_1 \ x_3 \, x_5 \, x_7 + x_3 \, x_4 \, x_5 \, x_7
       + x_1 \, x_2 \, x_6 \, x_7}{
       x_2 \, x_4 \, x_6}
     \\[2ex]
     \displaystyle
     \frac{
       x_3 \, x_4 \, x_5 + x_3 \, x_5 \, x_7 + x_2 \, x_6 \, x_7}{
       x_4 \, x_6}
     \\[2ex]
     x_3 
     \\
     x_7
   \end{pmatrix}^\top ,
   \notag
%   \allowdisplaybreaks[4]
  \\[2mm]
  &
  \mathsf{R}^{-1}(\boldsymbol{x})
  =
  \begin{pmatrix}
    x_1
    \\[1ex]
    \displaystyle
    \frac{x_1 \, x_3 \, x_5 +
      x_1 \,  x_2 \,  x_6 + x_2 \, x_4 \,  x_6
    }{
      x_3 \,  x_4}
    \\[2ex]
    x_6
    \\[1ex]
    \displaystyle
    \frac{
      x_1\, x_2 \,  x_4 \, x_6 + x_2 \, x_4^{~2} \, x_6
      + x_1 \,  x_3 \,  x_5 \,  x_7
      + x_1 \,  x_2 \,  x_6 \,  x_7 + 
      x_2 \, x_4 \, x_6 \,  x_7
    }{
      x_3 \,  x_4 \, x_5}
    \\[2ex]
    x_2
    \\[1ex]
    \displaystyle
    \frac{
      x_2 \,  x_4 \,  x_6
      + x_3 \, x_5 \,  x_7 + x_2 \,  x_6 \, x_7
    }{
      x_4 \, x_5}
    \\[2ex]
    x_7
  \end{pmatrix}^\top .
  \label{R_on_x}
\end{align}
%\end{equation}

Correspondingly, 
actions of the $\mathsf{R}$-operator,~\eqref{R} and~\eqref{Rinv},
on the $y$-variable are respectively given as follows:
  \begin{align}
    &
    \mathsf{R}(\boldsymbol{y})
    =
    \begin{pmatrix}
      \displaystyle
      y_1 \, \left(
        1 + y_2 + y_2 y_4
      \right) 
      \\[1ex]
      \displaystyle
      \frac{ y_2 \, y_4 \, y_5 \,  y_6}{
        1+ y_2 + y_6 + y_2 \, y_6 + y_2\, y_4\, y_6}
      \\[2ex]
      \displaystyle
      \frac{1+y_2+y_6+y_2 \, y_6 + y_2 \, y_4 \, y_6}{
        y_2 \, y_4}
      \\[2ex]
      \displaystyle
      \frac{y_4}{
        (1+y_2+y_2 \, y_4) \, ( 1+y_6+y_4 \, y_6)
      }
      \\[2ex]
      \displaystyle
      \frac{1+y_2+y_6+y_2 \, y_6 + y_2 \, y_4 \, y_6}{
        y_4 \, y_6}
      \\[2ex]
      \displaystyle
      \frac{ y_2 \, y_3 \, y_4 \, y_6}{
        1+ y_2 + y_6 + y_2 \, y_6 + y_2\, y_4 \, y_6}
      \\[2ex]
      \displaystyle
      \left(
        1+y_6+y_4 \, y_6
      \right) \, y_7
    \end{pmatrix}^\top ,
    \notag
%   \displaybreak[3]
%    \allowdisplaybreaks[4]
    \\[2mm]
    &
    \mathsf{R}^{-1}(\boldsymbol{y})
    =
    \begin{pmatrix}
      \displaystyle
      \frac{y_1 \, y_3 \,  y_4}{1 + y_4 + y_3 \, y_4}
      \\[3ex]
      \displaystyle
      \frac{y_5}{
        1 + y_4 + y_3 \, y_4 + y_4 \,  y_5 +   y_3 \,  y_4 \,  y_5}
      \\[2ex]
      \left(
        1 + y_4 + y_3 \,  y_4 + y_4 \,  y_5 + y_3 \,  y_4 \,  y_5
      \right) \, y_6
      \\[1ex]
      \displaystyle
      \frac{(1 + y_4 +    y_3 \, y_4) (1 + y_4 + y_4 \, y_5)}{
        y_3\, y_4  \,  y_5}
      \\[2ex]
      y_2  \, \left(
        1 + y_4 + y_3 \, y_4 + y_4 \,  y_5 + y_3 \,  y_4 \,  y_5
      \right)
      \\[1ex]
      \displaystyle
      \frac{y_3}{
        1 + y_4 + y_3 \,  y_4 + y_4 \,  y_5 + y_3 \,  y_4 \,  y_5}
      \\[3ex]
      \displaystyle
      \frac{y_4 \,  y_5 \,  y_7}{1 + y_4 + y_4 \,  y_5}
    \end{pmatrix}^\top .
  \label{R_on_y}
\end{align}

It should be remarked that
the $\mathsf{R}$-operator~\eqref{R} can be also written as
\begin{equation}
  \label{R_Jones}
  \mathsf{R} =
  s_{2,5} \, s_{3,6} \,
  \mu_2 \, \mu_6 \, \mu_4 \, \mu_2 \, \mu_6 ,
\end{equation}
which can be checked from
\begin{gather*}
  s_{3,5} \, \left( \mu_3 \, \mu_5 \, \mu_4 \right)^3
  = 1 ,
  \\
  s_{2,6} \, \left( \mu_2 \, \mu_6 \, \mu_4 \right)^3 = 1.
\end{gather*}
These identities  correspond to
a (half) periodicity in the cluster algebra associated to
$A_3$-type quiver, which is a sub-quiver of Fig.~\ref{fig:quiverR}.
See, \emph{e.g.},~\cite{FominZelev02a,Nakanishi10}.

%%%
\subsection{Braid Relation}

We generalize the quiver in Fig.~\ref{fig:quiverR} to that  in
Fig.~\ref{fig:quiver_generic}.
Therein also given is the triangulated disk,
and an exchange matrix $\mathbf{B}$ 
is given by the rule~\eqref{b_and_arrow}
as  a generalization of~\eqref{B_matrix_7}.

\begin{figure}[tbhp]
    \centering
    \includegraphics[]{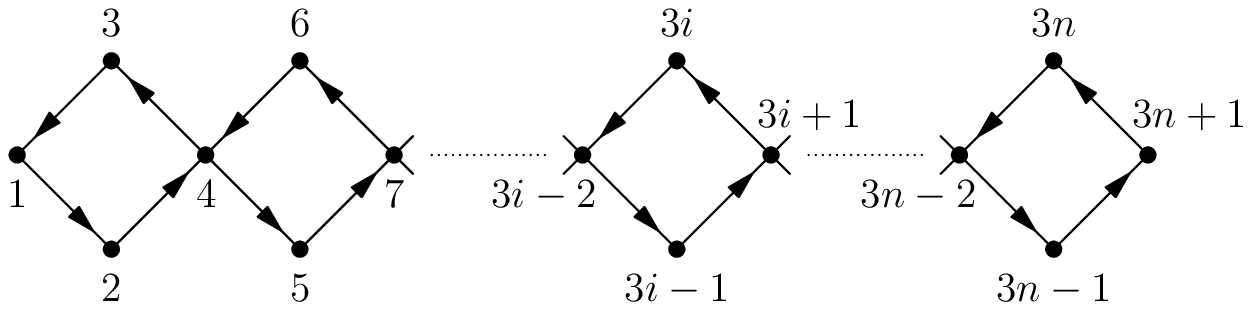}

    \bigskip

    \includegraphics[]{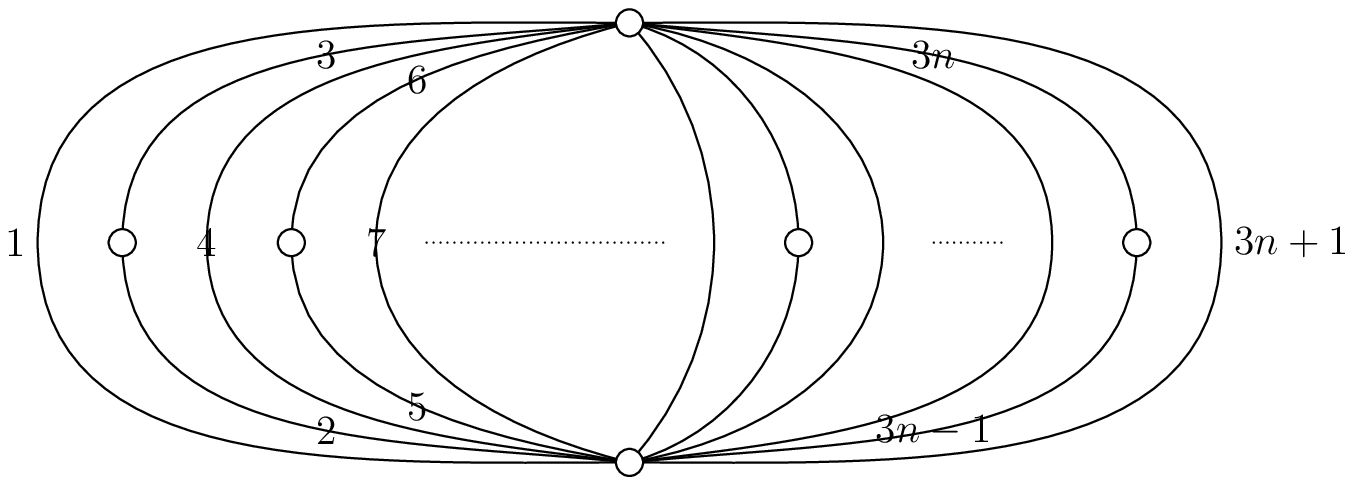}
  \caption{Quiver and triangulated disk.}
  \label{fig:quiver_generic}
\end{figure}

\begin{definition}
As a generalization of~\eqref{R},
we define the $\mathsf{R}$-operator
$\overset{i}{\mathsf{R}}$ for $i=1,\dots, n-1$ associated with the
quiver in Fig.~\ref{fig:quiver_generic} by
\begin{equation}
  \label{i-R}
  \begin{aligned}[b]
    \overset{i}{\mathsf{R}}
    & =
    s_{3i,3i+2} \, s_{3i-1,3i+2} \, s_{3i,3i+3} \,
    \mu_{3i+1} \, \mu_{3i-1} \, \mu_{3i+3} \, \mu_{3i+1} .
  \end{aligned}
\end{equation}
\end{definition}

Note that 
\begin{equation}
    \overset{i\phantom{-1}}{{\mathsf{R}^{\smash{-1}}}}
     =
    s_{3i, 3i+3} \,  s_{3i-1,3i+2} \, s_{3i,3i+2} \,
    \mu_{3i+1} \, \mu_{3i+2} \, \mu_{3i} \, \mu_{3i+1}  .
\end{equation}
The exchange matrix associated to Fig.~\ref{fig:quiver_generic}
is invariant under the action of
the $\mathsf{R}$-operators~$\overset{i\phantom{\pm1}}{{\mathsf{R}^{\smash{\pm 1}}}}$.
The explicit forms of the actions on
the cluster variable
$\boldsymbol{x}=(x_1, x_2, \dots, x_{3n+1})$ and
the $y$-variable
$\boldsymbol{y}=(y_1, y_2,  \dots, y_{3n+1})$
are as follows.
\begin{gather}
  \label{i-R_on_x}
  \overset{i\phantom{\pm1}}{\mathsf{R}^{\smash{\pm 1}}}(\boldsymbol{x})
  =
  \left(
    x_1, \dots, x_{3i-3},
    \mathsf{R}^{\pm 1}(x_{3i-2}, \dots, x_{3i+4}),
    x_{3i+5}, \dots, x_{3n+1}
  \right) ,
  \\[2ex]
  \label{i-R_on_y}
  \overset{i\phantom{\pm1}}{\mathsf{R}^{\smash{\pm 1}}}(\boldsymbol{y})
  =
  \left(
    y_1, \dots, y_{3i-3},
    \mathsf{R}^{\pm 1}(y_{3i-2}, \dots, y_{3i+4}),
    y_{3i+5}, \dots, y_{3n+1}
  \right) ,
\end{gather}
where $\mathsf{R}^{\pm1}(x_1,\dots,x_7)$ and $\mathsf{R}^{\pm1}(y_1,\dots,y_7)$ are
defined in~\eqref{R_on_x} and~\eqref{R_on_y} respectively.

\begin{theorem}
  \label{thm:main}
  The $\mathsf{R}$-operator satisfies the braid relation, namely
  we have
\begin{gather}
  \overset{i}{\mathsf{R}} \, \overset{i+1}{\mathsf{R}} \, \overset{i}{\mathsf{R}}
  =
  \overset{i+1}{\mathsf{R}} \, \overset{i}{\mathsf{R}} \, \overset{i+1}{\mathsf{R}} ,
  \quad
  \text{for $i=1,2,\dots,n-2$,}
  \\[2mm]
  \overset{i}{\mathsf{R}} \, \overset{j}{\mathsf{R}}
  =
  \overset{j}{\mathsf{R}} \, \overset{i}{\mathsf{R}},
  \quad
  \text{for $|i-j|>1$.}
\end{gather}
\end{theorem}

\begin{proof}
  The second equality is trivial.
  
  It is sufficient to check 
  $\overset{1}{\mathsf{R}} \,\overset{2}{\mathsf{R}}
  \,\overset{1}{\mathsf{R}}
  =
  \overset{2}{\mathsf{R}} \,\overset{1}{\mathsf{R}}
  \,\overset{2}{\mathsf{R}}$
  on the cluster variable  $(x_1,\dots,x_{10})$
  with the exchange matrix associated to Fig.~\ref{fig:quiver_generic}
  with $n=3$.
  By direct computation, 
  we can check that
  both actions,
  $\overset{1}{\mathsf{R}} \,\overset{2}{\mathsf{R}} \,\overset{1}{\mathsf{R}}$
  and
  $\overset{2}{\mathsf{R}} \,\overset{1}{\mathsf{R}} \,\overset{2}{\mathsf{R}}$,
  result
  in the following  same expressions,
  \begin{multline*}
    \Biggl(
    x_1, x_8,
    \frac{
      x_1 x_2 x_4 x_6 x_8 + x_1 x_3 x_5 x_7 x_8 + 
      x_3 x_4 x_5 x_7 x_8 + x_1 x_2 x_6 x_7 x_8 + x_1 x_2 x_4 x_5 x_9
    }{
      x_2 x_4 x_5 x_7} 
    ,
    \\
    \frac{1}{x_2 x_4 x_5 x_7 x_9} \,
    \bigl(
      x_1 x_2 x_4 x_6 x_7 x_8 + 
      x_1 x_3 x_5 x_7^{~2} x_8 + x_3 x_4 x_5 x_7^{~2} x_8 + 
      x_1 x_2 x_6 x_7^{~2} x_8 + x_1 x_2 x_4 x_6 x_8 x_{10}  
      \\
      +
      x_1 x_3 x_5 x_7 x_8 x_{10} + x_3 x_4 x_5 x_7 x_8 x_{10} + 
      x_1 x_2 x_6 x_7 x_8 x_{10} + x_1 x_2 x_4 x_5 x_9 x_{10}
    \bigr),
    \\
    \frac{
      x_6 x_7 x_8 + x_6 x_8 x_{10} + x_5 x_9 x_{10}
    }{x_7 x_9} 
    ,
    \frac{
      x_1 x_3 x_5 + x_3 x_4 x_5 + x_1 x_2 x_6
    }{x_2 x_4}
    ,
    \\
    \frac{1}{x_2 x_4 x_6 x_7 x_9}
    \bigl(
      x_1 x_3 x_4 x_6 x_7 x_8 + 
      x_3 x_4^{~2} x_6 x_7 x_8 + x_1 x_3 x_4 x_6 x_8 x_{10} + 
      x_3 x_4^{~2} x_6 x_8 x_{10} + x_1 x_3 x_4 x_5 x_9 x_{10}
      \\
      + 
      x_3 x_4^{~2} x_5 x_9 x_{10} + x_1 x_3 x_5 x_7 x_9 x_{10} + 
      x_3 x_4 x_5 x_7 x_9 x_{10} + x_1 x_2 x_6 x_7 x_9 x_{10}
    \bigr),
    \\
    \frac{
      x_3 x_4 x_6 x_7 x_8 + x_3 x_4 x_6 x_8 x_{10} + 
      x_3 x_4 x_5 x_9 x_{10} + x_3 x_5 x_7 x_9 x_{10} + 
      x_2 x_6 x_7 x_9 x_{10}
    }{
      x_4 x_6 x_7 x_9
    }     ,
    x_3, x_{10}
    \Biggr) .
  \end{multline*}
  Actions on the $y$-variables
  are induced from these actions.
  This completes the proof.
\end{proof}

The $\mathsf{R}$-operator~\eqref{R} is not new.
In~\cite{RKasha01a}
a solution of the Yang--Baxter equation is constructed from the
quantum dilogarithm function based on a relationship with the
Teichm{\"u}ller theory.
An operator, which has a similar action on the
$y$-variable~\eqref{R_on_y},
was used
in studies of lamination~\cite{Dynnik02a}.
In our case, the braiding denotes
an exchange of the  punctures
on the disk.
Also an operator which has a tropicalized action of cluster
variable~\eqref{R_on_x}
was given in~\cite{DehDynRolWie08}.
See~\cite{FockGonc07b} for applications of Teichm{\"u}ller coordinates
to laminations.

%%% 
\section{Hyperbolic Geometry}

\subsection{Ideal Tetrahedron}
A building block of hyperbolic 3-manifold is an ideal
tetrahedron whose vertices are on the boundary of a hyperbolic 
3-space~\cite{WPThurs80Lecture}.
An ideal hyperbolic tetrahedron $\triangle$ is parameterized with
cross-ratio $z \in \mathbb{C}$ of its four vertices, 
and the volume of $\triangle$ is given by the Bloch--Wigner function,
\begin{equation}
  D(z)
  =
  \Im \Li(z) + \arg(1-z) \, \log \left| z \right| .
\end{equation}
As  depicted in Fig.~\ref{fig:tetrahedron},
opposite edges have the same dihedral angles.
Therein we have used notations,
\begin{align}
  &z^\prime = 1- \frac{1}{z} ,
  &
  & z^{\prime \prime} = \frac{1}{1-z}  .
\end{align}

\begin{figure}[tbhp]
  \begin{minipage}{0.5\linewidth}
    \centering
    \includegraphics{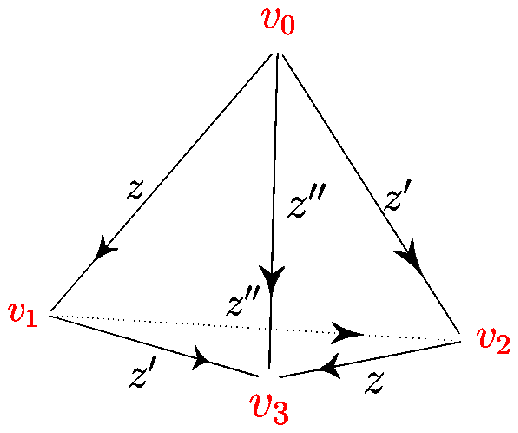}
  \end{minipage}
  \hfill
  \begin{minipage}{0.35\linewidth}
    \centering
    \includegraphics{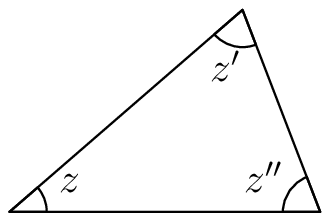}
  \end{minipage}
  \caption{An oriented ideal tetrahedron (left), and a triangle as intersection
    with horosphere (right).}
%    Vertex ordering of tetrahedron
  \label{fig:tetrahedron}
\end{figure}

In case that a set of ideal tetrahedra $\{ \triangle_{\nu} \}$ is
glued faces together to a hyperbolic manifold 
$M= \bigcup_\nu \triangle_{\nu}$,
the volume of $M$ is given by
\begin{equation}
  \Vol(M) = \sum_\nu D(z_\nu) .
\end{equation}

A complexification of $\Vol(M)$
known as a complex volume is defined 
with the Chern-Simons invariant $\CS(M)$. We have~\cite{WDNeum00a}
\begin{equation}
  \I \left( \Vol(M) + \I \CS(M) \right)
  =
  \sum_\nu \sign(\triangle_\nu) \, 
  \Rogers\left([z_\nu; p_\nu, q_\nu]\right)  ,
\end{equation}
where
$[z_\nu; p_\nu, q_\nu]$
is an element of the extended Bloch group with integers
$p_\nu$ and  $q_\nu$, and $\sign(\triangle_\nu)$ is 
$+1$ (resp. $-1$) when the vertex ordering of $\triangle_\nu$
is same (resp. inverse) with Fig.~\ref{fig:tetrahedron}. 
We have used the extended Rogers dilogarithm function
\begin{equation}
  \label{extended_Rogers}
  \Rogers\left([z;p,q]\right) =
  \Li(z) + \frac{1}{2} \log z \log(1-z)
  + \frac{\pi \, \I}{2} \left(
    q \log z + p \log(1-z)
  \right)
  - \frac{\pi^2}{6} .
\end{equation}

A method to compute
$p_\nu$ and $q_\nu$
% complex volume of $M$,
% $\Vol(M) + \I \CS(M)$,
was proposed in~\cite{Zicke09}.
For an oriented ideal tetrahedron of modulus $z$ 
in Fig.~\ref{fig:tetrahedron},
let $c_{ab}$ be complex parameters 
on edge connecting vertices~$v_a$ and~$v_b$.
Assume that they fulfill
\begin{align}
  \label{Zickert_c}
  &
  \frac{c_{03} \, c_{12}}{c_{02} \, c_{13}}
  = \pm z,
  &
  &
  \frac{c_{01} \, c_{23}}{c_{03} \, c_{12}}
  =
  \pm \left( 1 - \frac{1}{z} \right),
  &
  &
  \frac{c_{02} \, c_{13}}{c_{01} \, c_{23}}
  =
  \pm \frac{1}{1-z} .
\end{align}
Note that, in gluing tetrahedra together,
identical edges
have the same complex parameter.
Then 
$[z; p, q]$,
integers $p$ and $q$ for modulus $z$,
is given by
\begin{equation}
  \label{compute_pq}
  \begin{aligned}[b]
    \log z + p \, \pi \, \I 
    & =
    \log c_{03} + \log c_{12} - \log c_{02} - \log c_{13} ,
    \\
    -\log(1-z) +  q \, \pi \, \I 
    & =
    \log c_{02} + \log c_{13} - \log c_{01} - \log c_{23} .
  \end{aligned}
\end{equation}
Here and hereafter we mean the principal branch in the logarithm.
In~\cite{Zicke09}  these edge parameters $c_{ab}$ are read from a
developing map.

%%%%%%%%%%
\subsection{Octahedron}
In our previous paper~\cite{HikamiRInoue12a}, we 
demonstrated that the cluster mutation can be regarded as an
attachment of ideal tetrahedron to triangulated surface
(see also~\cite{NagaTeraYama11a}).
Furthermore
we claimed that
the cluster variable  $\boldsymbol{x}$ corresponds to Zickert's complex
parameters $c_{ab}$ on edges
(see~\cite[\S 2.3]{HikamiRInoue12a} for detail).
Roughly speaking, this is due to that all mutations used in 
$\mathsf{R}^{\pm 1}(\boldsymbol{x})$
have a form of the Ptolemy relation,
$a \,c + b\, d = e\, f$, which is same with~\eqref{Zickert_c}.

\begin{figure}[tbhp]
  \begin{minipage}{0.47\linewidth}
    \centering
    \includegraphics{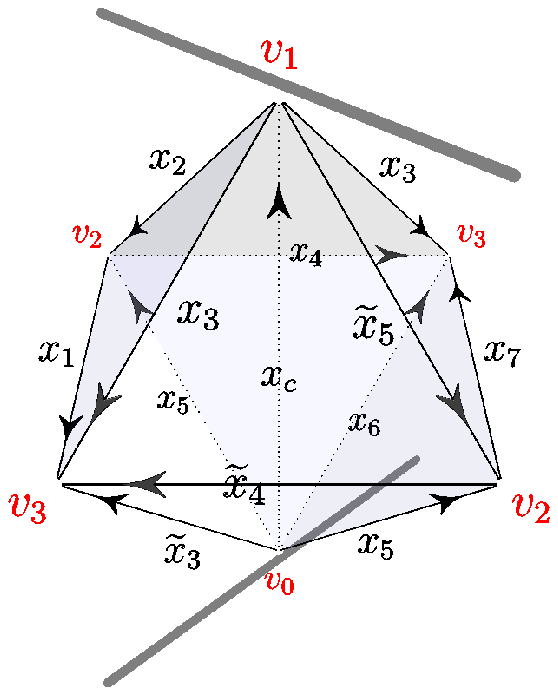}
  \end{minipage}
  \hfill
  \begin{minipage}{0.47\linewidth}
    \centering
    \includegraphics{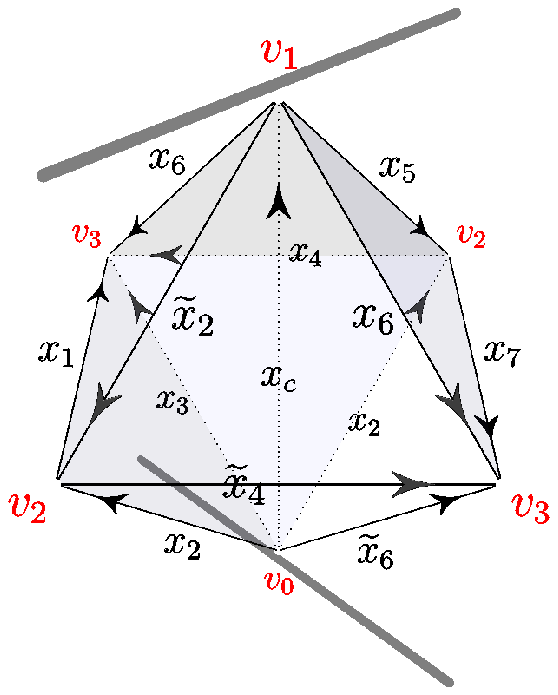}
  \end{minipage}
  \caption{Octahedron
  for $\overset{1}{\mathsf{R}}$ (left) and
  $\overset{1\hspace{12pt}}{\mathsf{R}^{-1}}$ (right)}
\label{fig:octahedron}
\end{figure}

% \begin{figure}[tbhp]
%   \centering
%   \includegraphics{octahedronRup}
% %  \includegraphics{octahedronRapart}
%   \caption{Octahedron}
%   \label{fig:octahedron_up}
% \end{figure}

For brevity,
we study a case
\begin{equation*}
  \widetilde{\boldsymbol{x}} =
  \overset{1\phantom{\pm 1}}{\mathsf{R}^{\pm 1}} ( \boldsymbol{x} ) .
\end{equation*}
Based on the  observation in~\cite{HikamiRInoue12a}, we see that the $\mathsf{R}$-operator~\eqref{R}
is realized as an octahedron in Fig.~\ref{fig:octahedron}, which is
composed of four tetrahedra
$\{\triangle_N, \triangle_S, \triangle_W, \triangle_E \}$.
See Fig.~\ref{fig:crossing_angle} for a top view.
The four tetrahedra
originate from
four mutations in the $\mathsf{R}^{\pm 1}$-operator,~\eqref{R}
and~\eqref{Rinv};
$\mu_2$ and $\mu_6$ in~\eqref{R} respectively correspond to $\triangle_W$ and
$\triangle_E$,
and two $\mu_4$'s are for $\triangle_N$ and $\triangle_S$.
The cluster variables $x_k$ and $\widetilde{x}_k$ are assigned to
edges of the octahedra,
and we have used
\begin{align}
  x_c & = 
  \frac{x_2 \, x_6 + x_3 \, x_5}{x_4} .
\end{align}
Note that
we have fixed vertex ordering  for our convention, and
that
edges with the same complex parameters
(\emph{e.g.}, two pairs of edges $v_0$--$v_2$, $v_1$--$v_3$)
are identical.

\begin{figure}[tbhp]
  \centering
    \begin{minipage}{0.47\linewidth}
    \centering
    \includegraphics{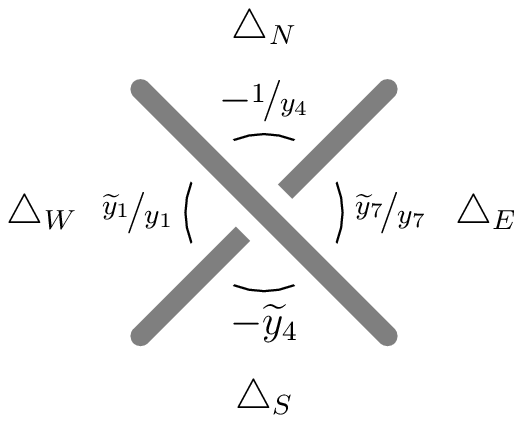}
  \end{minipage}
  \hfill
  \begin{minipage}{0.47\linewidth}
    \centering
    \includegraphics{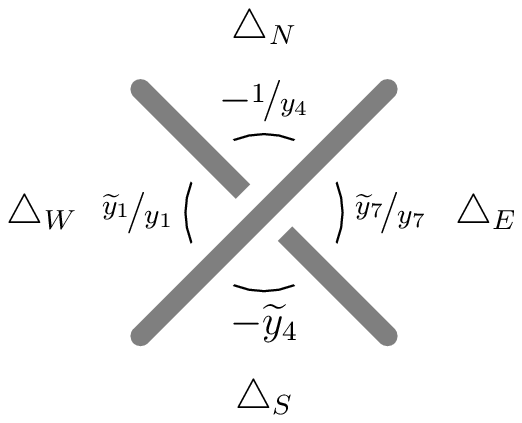}
  \end{minipage}
  \caption{Dihedral angle at crossings,
  $\overset{1}{\mathsf{R}}$ (left) and
  $\overset{1\hspace{12pt}}{\mathsf{R}^{-1}}$ (right).}
  \label{fig:crossing_angle}
\end{figure}

As the $\mathsf{R}$-operator satisfies the braid relation (Theorem~\ref{thm:main}),
we can interpret that each  octahedron is assigned to every
crossing of knot diagram as in Fig.~\ref{fig:crossing_angle}.
This 
reminds a fact~\cite{DThurs99a}
that
octahedron was assigned to the 
Kashaev $\mathsf{R}$-matrix~\cite{Kasha95}
(see also, \cite{Hikam00d,
%KHikami01e,
ChoKimKim13a,ChoMurYok09a,YYokota11a}).
%,HMuraka00c
Note that another expression~\eqref{R_Jones} of the same $\mathsf{R}$-operator
corresponds to a 
decomposition of octahedron into five tetrahedra,
which was used in studies of the colored Jones
$\mathsf{R}$-matrix at root of unity~\cite{DThurs99a,JChoMura10a}.

\begin{table}[tbhp]
  \newcolumntype{L}{>{$}l<{$}}
  \newcolumntype{R}{>{$\displaystyle}r<{$}}
  \newcolumntype{C}{>{$\displaystyle}c<{$}}
  \rowcolors{2}{gray!11}{}
  \centering
  \begin{tabular}{CCCC*{2}{R}CC*{2}{R}}
    \toprule
    \addlinespace[3pt]
    &&& \multicolumn{3}{C}{\overset{1}{\mathsf{R}}} &&
    \multicolumn{3}{C}{\overset{1\phantom{-1}}{\mathsf{R}^{-1}}} 
    \\
    \addlinespace[3pt]
    \cmidrule{4-6} \cmidrule{8-10}
    \triangle & \text{Volume} &
    \phantom{AA}&
    \sign(\triangle) &
    z_\triangle  \phantom{A} &  \frac{1}{1-z_\triangle} 
    &\phantom{AA}&
    \sign(\triangle) &
    z_\triangle  \phantom{A}&  \frac{1}{1-z_\triangle} 
    \\
    \midrule
    \addlinespace[3pt]
    \triangle_N &
    D
    \left(
      -  \frac{1}{y_4}
    \right)
    &&
    - &
    -\frac{x_2 \, x_6}{x_3 \, x_5} &
    \frac{x_3 \, x_5}{x_4 \, x_c} 
    &&
    + &
    - \frac{x_3 \, x_5}{x_2 \, x_6} &
    \frac{x_2 \, x_6}{x_4 \, x_c}
    \\ 
    \addlinespace[3pt]
    \triangle_S &
    D
    \left(
      -\widetilde{y}_4
    \right)
    &&
    - &
    -\frac{\widetilde{x}_3 \, \widetilde{x}_5}{x_3 \, x_5}
    &
    \frac{{x}_3 \, {x}_5}{
      \widetilde{x}_4 \,   {x}_c}
    &&
    + &
    - \frac{\widetilde{x}_2 \, \widetilde{x}_6}{x_2 \, x_6}
    &
    \frac{      x_2 \, {x}_6}{x_c \, \widetilde{x}_4}
    \\
    \addlinespace[3pt]
    \triangle_W &
    D
    \left(
      \frac{\widetilde{y}_1}{y_1}
    \right)
    &&
    + &
    \frac{x_2 \, \widetilde{x}_3}{x_3 \, x_5}
    &
    -\frac{x_3 \, x_5}{x_1 \, x_c}
    &&
    - &
    \frac{\widetilde{x}_2 \, x_3}{x_2 \, x_6}
    &
    -\frac{{x}_2 \, x_6}{x_1 \, x_c}
    \\
    \addlinespace[3pt]
    \triangle_E &
    D
    \left(
      \frac{\widetilde{y}_7}{y_7}
    \right)
    &&
    + &
    \frac{\widetilde{x}_5 \, {x}_6}{x_3 \, x_5}
    &
    -\frac{x_3 \, x_5}{x_c \, x_7}
    &&
    - &
    \frac{x_5 \, \widetilde{x}_6}{x_2 \, x_6}
    &
    -\frac{x_2 \, {x}_6}{{x}_c \, {x}_7}
    \\
    \bottomrule
  \end{tabular}
  \bigskip
  \caption{Moduli of four tetrahedra assigned to 
    operators~$\overset{1}{\mathsf{R}}$ and
    $\overset{1\hspace{12pt}}{\mathsf{R}^{-1}}$.
    Sgn ``$+$'' 
  (resp. ``$-$'')
  means that vertex ordering of tetrahedron is
  same
  (resp. inverse)
  with Fig.~\ref{fig:tetrahedron}.}
  \label{tab:modulus_R}
\end{table}

Taking into account  of the vertex ordering of tetrahedra, we can 
determine moduli of each tetrahedron
from~\eqref{Zickert_c} as in Table~\ref{tab:modulus_R}.
From these results,
we define dilogarithm functions for every crossing by
\begin{equation}
  \label{eq:3}
  \Rogers([\overset{1\phantom{\pm 1}}{\mathsf{R}^{\pm 1}}];
  \boldsymbol{x})
  =
  \sum_{t \in \{ N,S,W,E \}}
  \sign(\triangle_t) \,
  \Rogers \left(
    [z_{\triangle_t} ; p_{\triangle_t} , q_{\triangle_t} ]
  \right) .
\end{equation}
Here we have used the extended Rogers dilogarithm~\eqref{extended_Rogers}, and
integers $p_{\triangle_t}$ and $q_{\triangle_t}$ are given
from~\eqref{compute_pq} by use of Table~\ref{tab:modulus_R}.
For instance,
$p_{\triangle_E}$ and $q_{\triangle_E}$
in the operator~$\overset{1}{\mathsf{R}}$
are given as
\begin{equation*}
  \begin{aligned}
    p_{\triangle_E} \, \pi \, \I
    &=
    - \log \left(
      \frac{
        \widetilde{x}_5 \, x_6}{x_3 \, x_5}
    \right)
    + \log({\widetilde{x}_5}) + \log (x_6) - \log(x_3) - \log(x_5),
    \\[1ex]
    q_{\triangle_E} \, \pi \, \I
    & =
    - \log \left(
      - \frac{x_3 \, x_5}{x_c \, x_7}
    \right)
    +
    \log(x_3) + \log(x_5) - \log(x_c) - \log(x_7) .
  \end{aligned}
\end{equation*}

It should be remarked that,
to identify the $\mathsf{R}$-operator with a hyperbolic octahedron,
we need a consistency condition around a central edge labeled by
$x_c$ in Fig.~\ref{fig:octahedron}.
This condition is automatically satisfied by
\begin{equation*}
  y_1 \, y_4 \, y_7 =
  \widetilde{y}_1 \,  \widetilde{y}_4 \,  \widetilde{y}_7 ,
\end{equation*}
where
$\widetilde{\boldsymbol{y}} 
  = \overset{1\phantom{\pm 1}}{\mathsf{R}^{\pm
  1}}(\boldsymbol{y})$~\eqref{i-R_on_y}.
In Fig.~\ref{fig:crossing_angle} 
denoted are
dihedral angles around central axis assigned to each crossing.

The $i$-th braiding operator
$\overset{i\phantom{\pm 1}}{\mathsf{R}^{\pm 1}}$
in~\eqref{i-R} can be interpreted in
the same manner.
As we have the cluster mutation
$\widetilde{\boldsymbol{x}}
=
\overset{i\phantom{\pm 1}}{\mathsf{R}^{\pm 1}}(\boldsymbol{x})
$
as in~\eqref{i-R_on_x},
the edge parameters 
$(x_1, x_2, \dots , x_7)$
and
$(\widetilde{x}_1, \widetilde{x}_2, \dots , \widetilde{x}_7)$
in Fig.~\ref{fig:octahedron}
are
replaced respectively
by $(x_{3i-2}, x_{3i-1}, \dots, x_{3i+4})$
and
$(\widetilde{x}_{3i-2}, \widetilde{x}_{3i-1}, \dots, \widetilde{x}_{3i+4})$.
The moduli of the tetrahedra  in
Table~\ref{tab:modulus_R} should be replaced correspondingly,
and 
as a result
we have
the dilogarithm function
$  \Rogers([\overset{i\phantom{\pm 1}}{\mathsf{R}^{\pm 1}}];
  \boldsymbol{x})
$ as in~\eqref{eq:3} by replacing $x_a$ with ${x}_{3i+a-3}$.

%%%%
\subsection{Braid Group Presentation and Gluing Conditions}

Our main claim is the following.
\begin{theorem}
  \label{thm:main2}
  Let a knot $K$ have a braid group presentation
  $\sigma_{k_1}^{~\varepsilon_1} \,
  \sigma_{k_2}^{~\varepsilon_2} \cdots \sigma_{k_m}^{~\varepsilon_m}$,
  where $\varepsilon_j=\pm 1$ and
  \begin{equation*}
    \mathcal{B}_n=
    \left\langle
    \sigma_1, \sigma_2,  \dots, \sigma_{n-1}
    ~\middle|~
    \begin{array}{c}
      \text{$
        \sigma_i \sigma_j = \sigma_j \sigma_i$
        for $|i-j|>1$}
      \\
      \text{$
        \sigma_i \sigma_{i+1} \sigma_i=
        \sigma_{i+1} \sigma_{i} \sigma_{i+1}
        $ for $i=1,2,\dots,n-2$}
    \end{array}
    \right\rangle .
  \end{equation*}
  We define a cluster pattern
  for
  $\boldsymbol{x}[j]=
  \left(
    x[j]_1, \dots, x[j]_{3n+1} \right)$
  by
  \begin{equation}
    \boldsymbol{x}[1]
    \xrightarrow{\overset{k_1\phantom{\varepsilon}
      }{\mathsf{R}^{\smash{\varepsilon_1}}}}
    \boldsymbol{x}[2]
    \xrightarrow{\overset{k_2\phantom{\varepsilon}
      }{\mathsf{R}^{\smash{\varepsilon_2}}}}
    \cdots
    \xrightarrow{\overset{k_m\phantom{\varepsilon}
      }{\mathsf{R}^{\smash{\varepsilon_m}}}}
    \boldsymbol{x}[m+1] ,
  \end{equation}
  with the exchange matrix associated to
  Fig.~\ref{fig:quiver_generic}.
  We assume that the initial cluster variable $\boldsymbol{x}[1]$ 
  satisfies
  \begin{equation}
    \label{periodic_x}
    \boldsymbol{x}[1]= \boldsymbol{x}[m+1] .
  \end{equation}
  Then the $y$-variables, $y[k]_i \in \mathbb{C}$, 
  induced from the cluster pattern
  fulfill algebraic equations for shape parameters of 
  ideal tetrahedra in the triangulation of 
  $S^3 \setminus (K \cup \text{$2$-points})$. 
\end{theorem}

We note that the periodicity~\eqref{periodic_x} denotes a closure of
the braid,
and that
the $2$-points are
$v_2$ and $v_3$ in Fig.~\ref{fig:octahedron}.

In the above theorem, we do not assume that a knot~$K$
is hyperbolic.
We study a triangulation induced from
a braid group presentation.
This situation is same with the volume conjecture~\cite{Kasha96b},
which suggests an intimate relationship between
a complex volume of $S^3\setminus K$
and
the Kashaev invariant for $K$ defined from a quantum $R$-matrix.

% Unfortunately, at this stage, we do not know which solutions
% of~\eqref{periodic_x}
% %under~\eqref{cancel_x}
% are ``preferable''. 
%Although, 
Our triangulation is a standard one used
%same with that
in
\texttt{SnapPy}~\cite{Weeks05},
and the Neumann--Zagier potential function was
constructed in~\cite{ChoMurYok09a} from such triangulation.
See also \cite{InoueKabay14a}, where a complex volume is studied from
the same triangulation by use of quandle.
So it is natural to expect that
for hyperbolic knot~$K$
there exists a \emph{geometric solution}
of~\eqref{periodic_x},
% with~\eqref{cancel_x},
where
the neighbors of additional two points cancel and
we endow
a complete hyperbolic structure for
$S^3 \setminus K$.
We show in the next section
numerical results for some knots, and we discuss
how the cancellation of two balls occurs
(see Prop.~\ref{prop:cancel}).
Unfortunately, at this stage, we do not know how to extract
generally  such a
preferable
solution from~\eqref{periodic_x}.
Due to that the geometric content of each octahedron is identified as
in Table~\ref{tab:modulus_R},
we obtain complex volume as
%~\eqref{complex_from_R}
follows
if we assume an existence of
geometric solution.

% Under this assumption,
% we propose
% the following conjecture.
% %if we assume that there exists a ``good'' hyperbolic solution
% %``geometric'' solution of~\eqref{periodic_x},
% Due to that the geometric content of each octahedron is identified as
% in Table~\ref{tab:modulus_R},
% we obtain~\eqref{complex_from_R}.

\begin{conj}
  \label{conj}
  There exists  an algebraic solution of~\eqref{periodic_x} such
  that
  the complex volume of  $K$ is given
  by 
  \begin{equation}
    \label{complex_from_R}
    \I \left(
      \Vol(S^3 \setminus K)
      + \I \CS(S^3 \setminus K)
    \right)
    =
    \sum_{j=1}^m
    \Rogers(
    [ \overset{k_j \phantom{\varepsilon}}{\mathsf{R}^{\varepsilon_j}}
    ];
    \boldsymbol{x}[j]) .
  \end{equation}
  See~\eqref{eq:3}  and the end of the last subsection for
  the definition of
  the dilogarithm
  function
  $
    \Rogers(
    [ \overset{k_j \phantom{\varepsilon}}{\mathsf{R}^{\varepsilon_j}}
    ];
    \boldsymbol{x}[j]) 
  $.
\end{conj}

% As supports for Conjecture~\ref{conj}, we will give two examples 
% in the next section.

\begin{figure}[tbhp]
  \begin{minipage}{0.7\linewidth}
    \centering
    \includegraphics[scale=0.8]{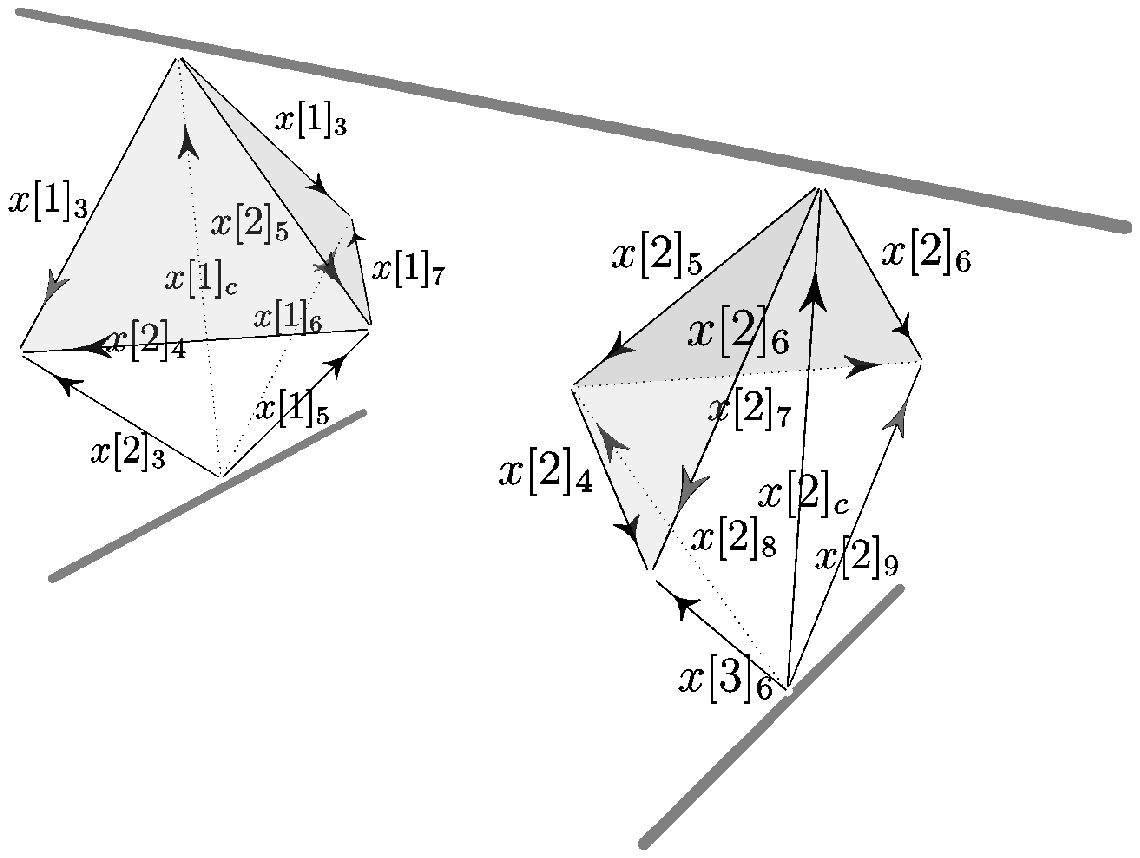}
  \end{minipage}
  \hfill
  \begin{minipage}{0.28\linewidth}
    \centering
    \includegraphics[scale=0.7]{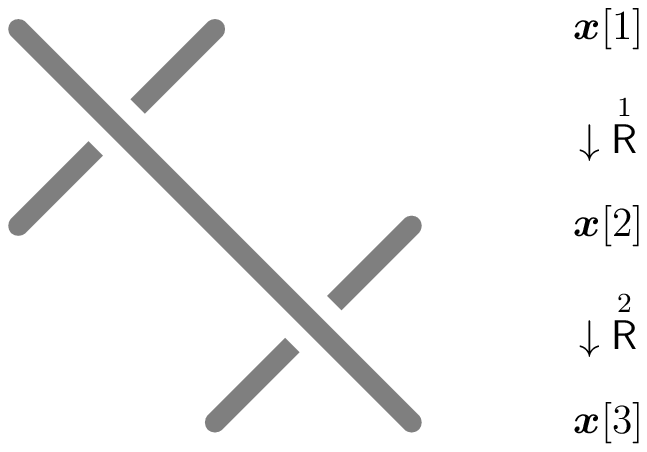}

    \bigskip    \bigskip

    \includegraphics[scale=0.8]{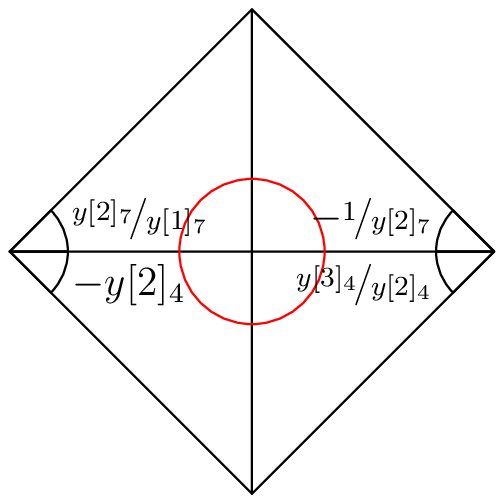}
  \end{minipage}
  \caption{Gluing of octahedra (left) assigned to
    crossing (right    top),
    and a developing map (right bottom).
    A consistency condition is read from a red circle.
  }
  \label{fig:glue1}
\end{figure}

% In the closing of this section, we present the proof of 
% Theorem~\ref{thm:main2}. 

\begin{proof}[Proof of Theorem~\ref{thm:main2}] 
We need to check consistency conditions and completeness conditions
as~\cite{Hikam00d,ChoKimKim13a}.
We have already seen that
a consistency condition around a central axis
of octahedra is fulfilled.
We shall check other cases.
First we study a cluster pattern
\begin{equation*}
  \boldsymbol{x}[1]  \xrightarrow[]{\overset{1}{\mathsf{R}}}
  \boldsymbol{x}[2]  \xrightarrow[]{\overset{2}{\mathsf{R}}}
  \boldsymbol{x}[3] .
\end{equation*}
For each crossing we assign octahedra as in Fig.~\ref{fig:glue1}.
Therein colored faces are glued together so that 
identical edges have the same complex parameters.
Note 
that~\eqref{R_on_x} implies
$x[2]_6=x[1]_3$ and $x[1]_7=x[2]_7$, and that
\begin{align*}
  &x[1]_c =
  \frac{
    x[1]_2 \, x[1]_6 + x[1]_3 \, x[1]_5}{
    x[1]_4},
  &
  &
  x[2]_c =
  \frac{
    x[2]_5 \, x[2]_9 + x[2]_6 \, x[2]_8}{
    x[2]_7} .
\end{align*}
Consistency condition around edge labeled by complex parameter
$x[2]_5$ is checked as
\begin{equation*}
  \frac{1}{
    1- \frac{y[2]_7}{y[1]_7}
  } \cdot
  \left( 1+ \frac{1}{y[2]_4} \right) \cdot
  \left( 1+ y[2]_7 \right) \cdot
  \frac{1}{
    1- \frac{y[3]_4}{y[2]_4}
  } = 1.
\end{equation*}

\begin{figure}[tbhp]
  \begin{minipage}{0.7\linewidth}
    \centering
    \includegraphics[scale=0.8]{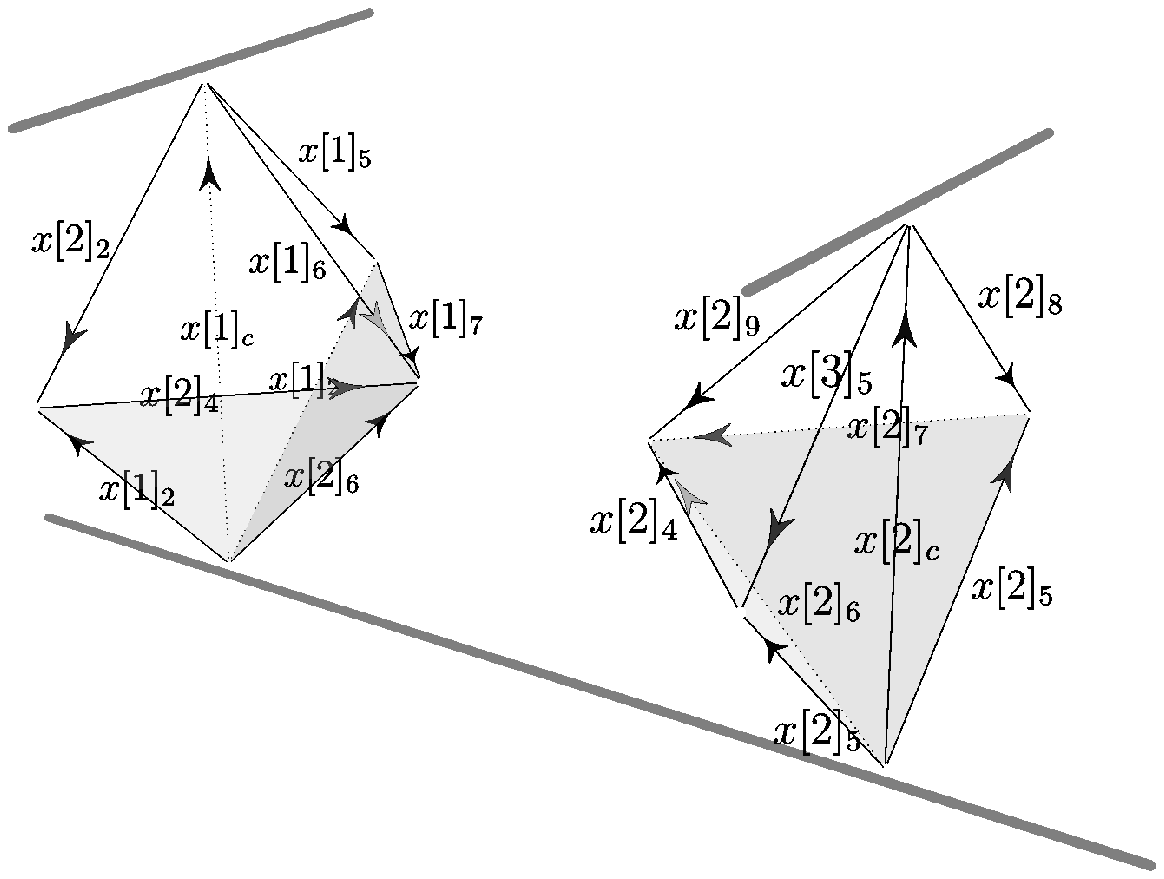}
  \end{minipage}
  \hfill
  \begin{minipage}{0.28\linewidth}
    \centering
    \includegraphics[scale=0.7]{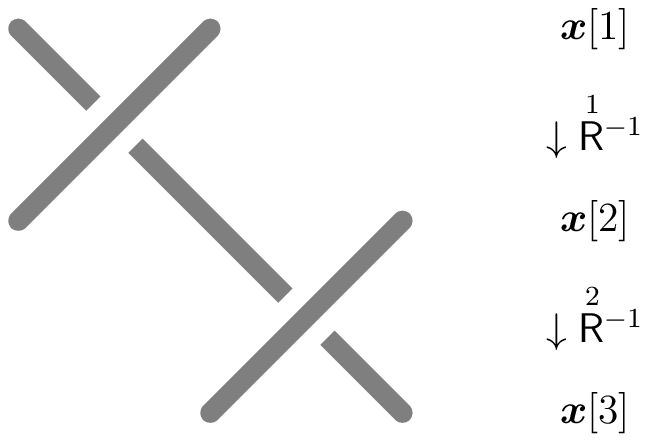}

    \bigskip    \bigskip

    \includegraphics[scale=0.8]{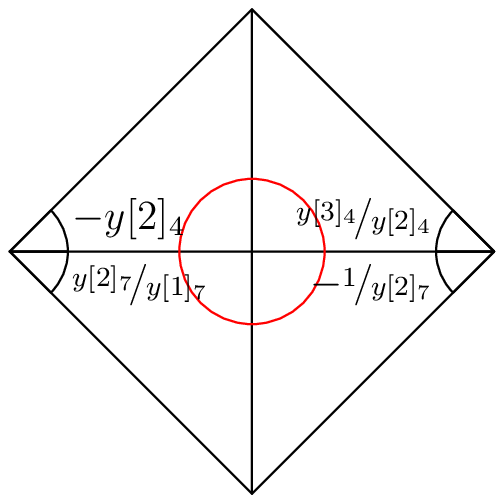}
  \end{minipage}
  \caption{Gluing of octahedra (left) assigned to crossing (right
    top), and a developing map (right bottom).
    A consistency condition is read from a  red  circle.}
  \label{fig:glue3}
\end{figure}

Same is  a case of  a cluster pattern
\begin{equation*}
  \boldsymbol{x}[1]  \xrightarrow[]{\overset{1\phantom{-1}}{\mathsf{R}^{\smash{-1}}}}
  \boldsymbol{x}[2]  \xrightarrow[]{\overset{2\phantom{-1}}{\mathsf{R}^{\smash{-1}}}}
  \boldsymbol{x}[3] .
\end{equation*}
We have octahedra as in Fig.~\ref{fig:glue3}, and 
we can check a consistency condition in the developing map as
\begin{equation*}
  \left( 1 - \frac{y[1]_7}{y[2]_7} \right) \cdot
  \frac{1}{1+y[2]_4} \cdot
  \left( 1 - \frac{y[2]_4}{y[3]_4} \right) \cdot
  \frac{1}{1+ \frac{1}{y[2]_7}} = 1.
\end{equation*}

\begin{figure}[tbhp]
  \begin{minipage}{0.7\linewidth}
    \centering
    \includegraphics[scale=0.8]{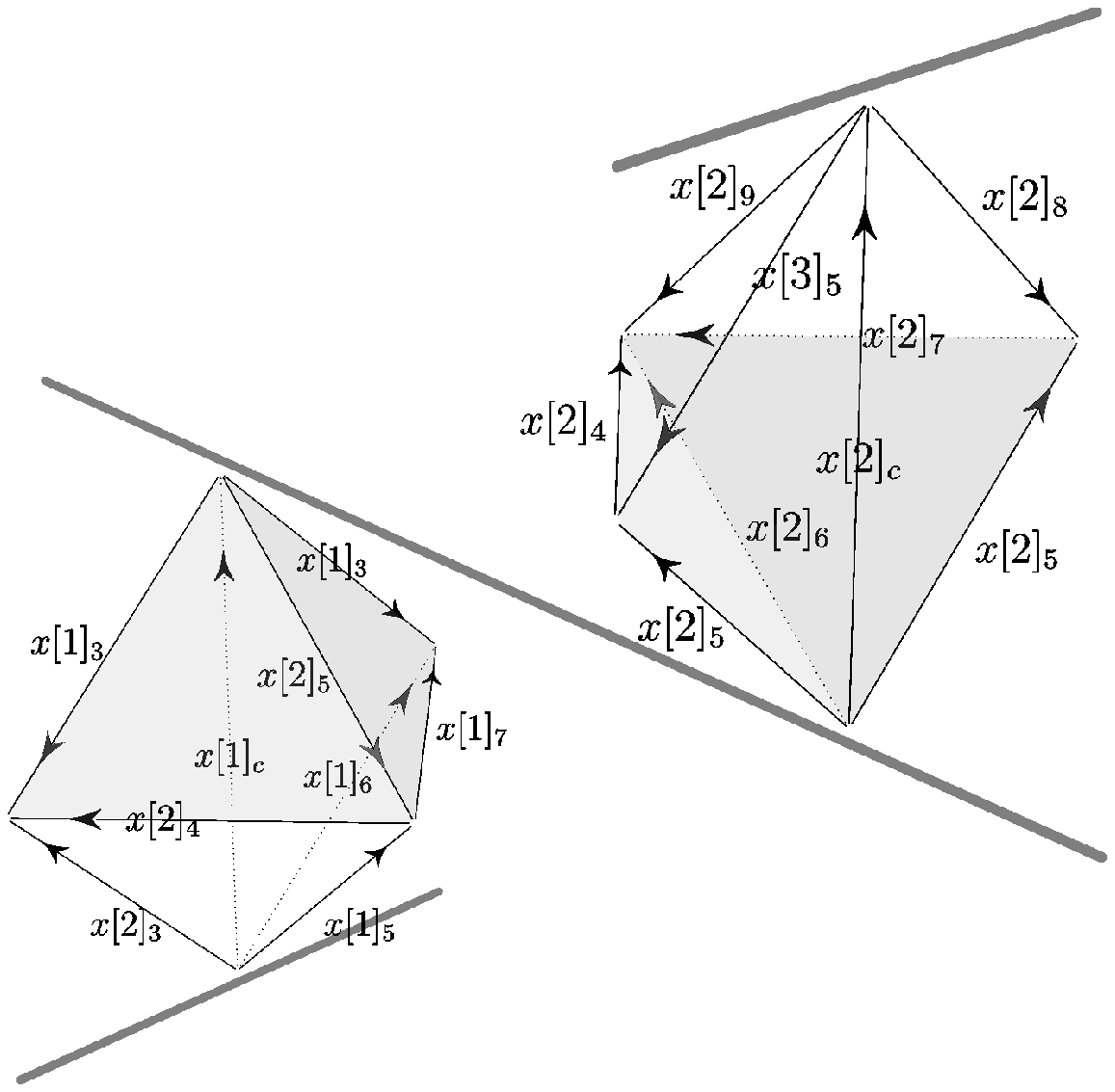}
  \end{minipage}
  \hfill
  \begin{minipage}{0.28\linewidth}
    \centering
    \includegraphics[scale=0.7]{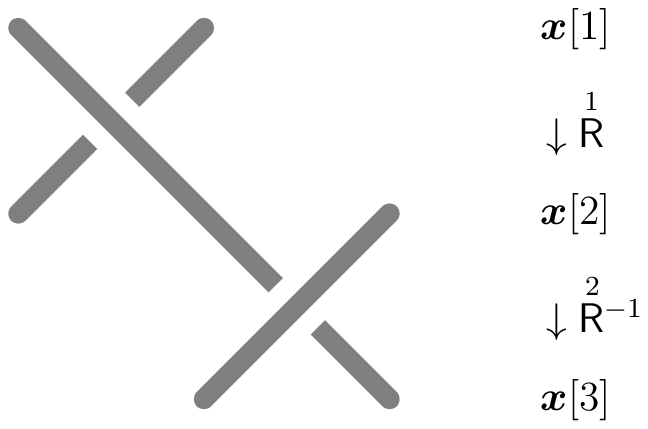}

    \bigskip    \bigskip

    \includegraphics[scale=0.8]{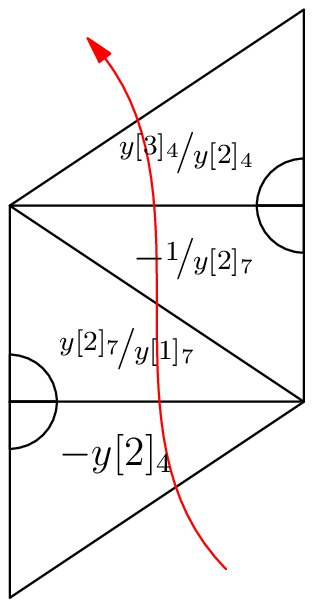}
  \end{minipage}
  \caption{Gluing of octahedra (left) assigned to
    crossing (right top), and
    a developing map (right bottom).
    A completeness condition is read from
    a red curve.}
  \label{fig:glue2}
\end{figure}

A completeness condition follows from alternating crossings.
In the case that the cluster pattern is given by
\begin{equation*}
  \boldsymbol{x}[1]  \xrightarrow[]{\overset{1}{\mathsf{R}}}
  \boldsymbol{x}[2]  \xrightarrow[]{\overset{2\phantom{-1}}{\mathsf{R}^{\smash{-1}}}}
  \boldsymbol{x}[3] ,
\end{equation*}
octahedra
are  attached to each crossing as in Fig.~\ref{fig:glue2}.
See that identical edges have same complex parameters
$x[2]_6=x[1]_3$ and $x[1]_7=x[2]_7$
due to~\eqref{R_on_x}.
Then we can check the completeness condition
as
\begin{equation*}
  \begin{aligned}[b]
    \frac{
      1+ \frac{1}{y[2]_4}}{
      1 - \frac{y[2]_7}{y[1]_7}
    } \cdot
    \frac{1 + \frac{1}{y[2]_7}}{
      1- \frac{y[2]_4}{y[3]_4}
    }
    & =
    y[1]_2 \, y[1]_3
    \\
    & = 1 .
  \end{aligned}
\end{equation*}
Here the last equality follows from~\eqref{y_from_x}.

\begin{figure}[tbhp]
  \begin{minipage}{0.7\linewidth}
    \centering
    \includegraphics[scale=0.8]{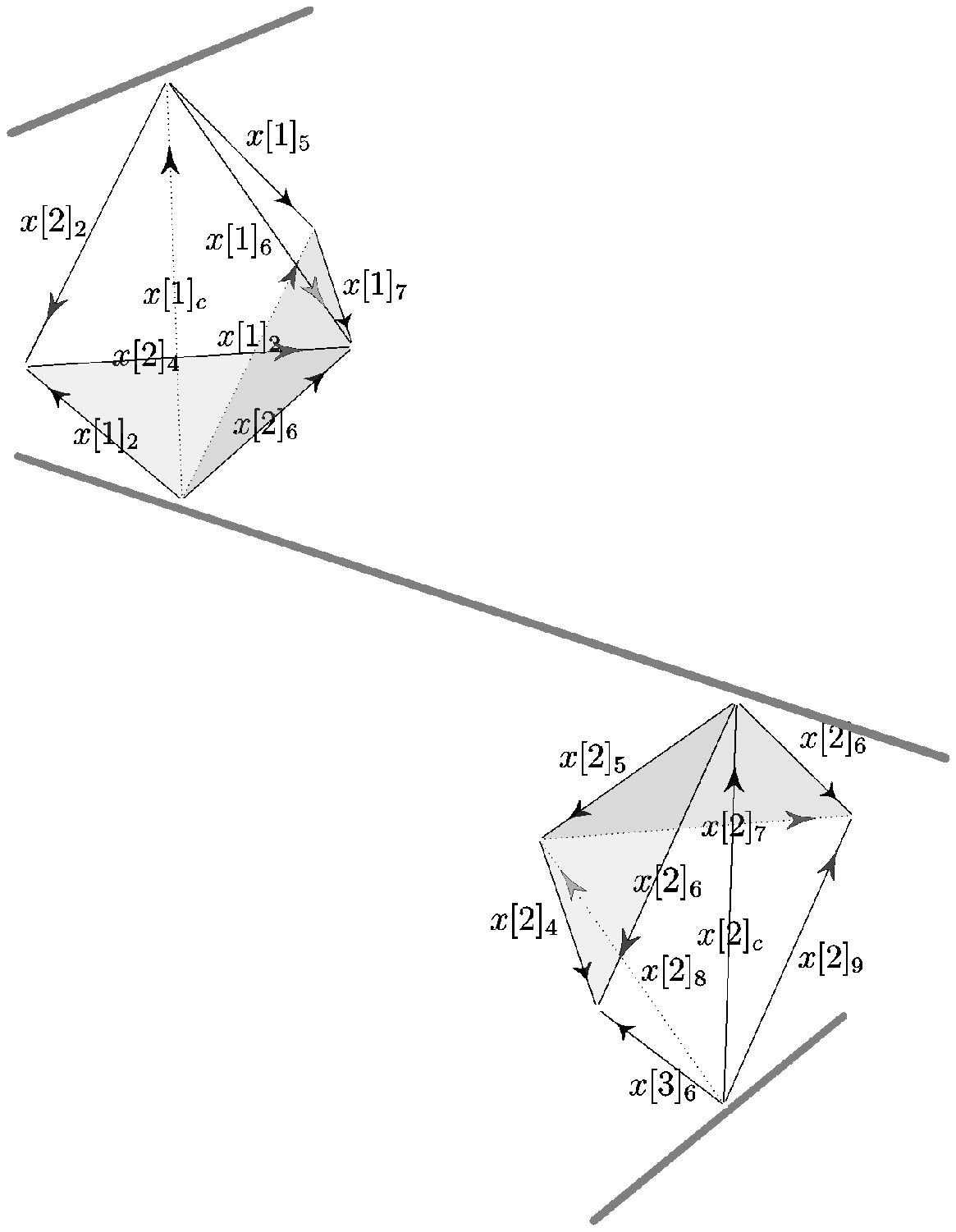}
  \end{minipage}
  \hfill
  \begin{minipage}{0.28\linewidth}
    \centering
    \includegraphics[scale=0.7]{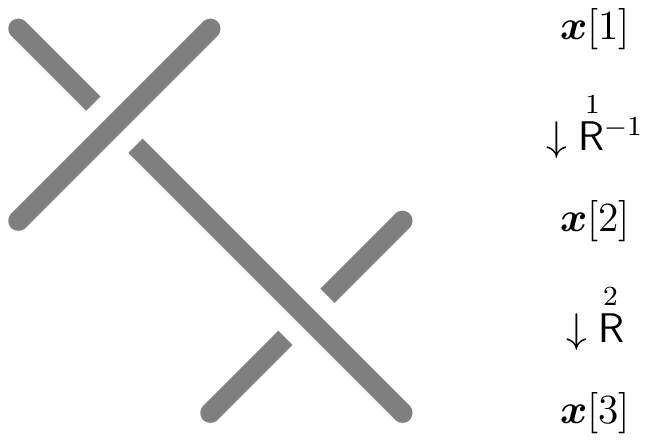}

    \bigskip    \bigskip
    
    \includegraphics[scale=0.8]{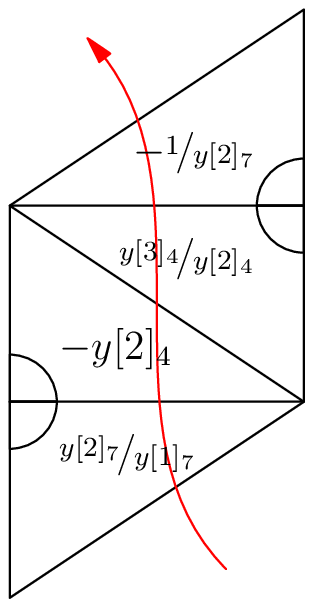}
  \end{minipage}
  \caption{Gluing of octahedra (left) assigned to crossing (right
    top),
    and a developing map (right bottom).}
  \label{fig:glue4}
\end{figure}

We have Fig.~\ref{fig:glue4} for a cluster pattern
\begin{equation*}
  \boldsymbol{x}[1]  \xrightarrow[]{\overset{1\phantom{-1}}{\mathsf{R}^{\smash{-1}}}}
  \boldsymbol{x}[2]  \xrightarrow[]{\overset{2}{\mathsf{R}}}
  \boldsymbol{x}[3] .
\end{equation*}
By use of~\eqref{y_from_x}, we have a completeness condition
\begin{equation*}
  \begin{aligned}[b]
    \frac{1- \frac{y[1]_7}{y[2]_7}}{1+y[2]_4} \cdot
    \frac{1 - \frac{y[3]_4}{y[2]_4}}{1+y[2]_7}
    & = y[1]_2 \, y[1]_3
    \\
    & = 1.
  \end{aligned}
\end{equation*}

Other cases can be checked in a similar manner,
and the claim follows.
\end{proof}

We note that in the above proof the completeness condition
is
% endowed by
\begin{equation}
  y[1]_{3i-1} \, y[1]_{3i} =1,
  \quad
  \text{for $i=1, 2, \dots, n$,}
\end{equation}
which follows from the definition of the $y$-variables \eqref{y_from_x}.

%%%%%%%%%
\section{Examples}

\subsection{\mathversion{bold}Figure-eight knot $4_1$}

We study the figure-eight knot whose braid group presentation is 
$\sigma_1 \, \sigma_2^{~-1} \, \sigma_1 \, \sigma_2^{~-1}$.
The cluster pattern for $4_1$ is thus
\begin{equation*}
  \boldsymbol{x}[1]
  \xrightarrow[]{\overset{1}{\mathsf{R}}}
  \boldsymbol{x}[2]
  \xrightarrow[]{\overset{2\phantom{-1}}{\mathsf{R}^{\smash{-1}}}}
  \boldsymbol{x}[3]
  \xrightarrow[]{\overset{1}{\mathsf{R}}}
  \boldsymbol{x}[4]
  \xrightarrow[]{\overset{2\phantom{-1}}{\mathsf{R}^{\smash{-1}}}}
  \boldsymbol{x}[5] .
\end{equation*}
We can check that
$\boldsymbol{x}[1]=\boldsymbol{x}[5]$ is fulfilled by, for example, 
\begin{equation*}
  \boldsymbol{x}[1]
  =
  \left(
    x_1, x_2, x_2, 1, x_1\, x_2, x_1^{~2}  \,x_2,
    x_1, -x_2, -x_2, 1
  \right),
\end{equation*}  
where $(x_1, x_2) = (\E^{2 \pi \I/3}, 0)$.
To compute the complex volume of $4_1$, we set
$(x_1, x_2) = (\E^{2\pi \I/3} + \delta, \delta)$
with $\delta \in \mathbb{R}_{>0}$,
and take a limit  $\delta\to 0$ at the last.
We  have checked numerically that~\eqref{complex_from_R} gives
$\I \cdot 2 \, D(\E^{\pi \I/3})
=
\I \cdot 2.02988 \cdots $ as desired~\cite{WPThurs80Lecture}.

%%%%
\subsection{\mathversion{bold} Trefoil knot $3_1$}

% We observe that Theorem~\ref{thm:main2} and Conjecture~\ref{conj}
% can be also applied to the trefoil knot $3_1$,
Next example is the trefoil $3_1$,
which is not hyperbolic.
The braid group presentation for $3_1$ is $\sigma_1^{~3}$,
and
its cluster pattern is 
\begin{equation*}
  \boldsymbol{x}[1]
  \xrightarrow[]{\overset{1}{\mathsf{R}}}
  \boldsymbol{x}[2]
  \xrightarrow[]{\overset{1}{\mathsf{R}}}
  \boldsymbol{x}[3]
  \xrightarrow[]{\overset{1}{\mathsf{R}}}
  \boldsymbol{x}[4].
\end{equation*}
We solve $\boldsymbol{x}[1] = \boldsymbol{x}[4]$ by choosing 
an initial cluster variable as
\begin{align}
  \label{3_1-initial}
  \boldsymbol{x}[1] 
  =
%  (1,x,-x,1,-x^{2}, 1,1),
  (x_1, x_2, x_2, 1,  x_1 x_2, x_1^{~2} x_2, 1)
\end{align}
and get %$x=\frac{1 \pm \I \sqrt{15}}{4}$.
$x_1=-\frac{1+\I}{2}$ in a limit $x_2 \to 0$.
We check numerically that ~\eqref{complex_from_R} gives
$-8.22467 \cdots \simeq -\frac{5}{6} \pi^2$.
It agrees with the Chern-Simons invariant of $3_1$,
which is also given from asymptotic limit of the Kashaev
invariant~\cite{KashTirk00,Zagier01,HikamiKirillov03}.

%%%%
\subsection{Interpretation of Initial Cluster Variables}
In the above examples, we have   singular solutions
such as
\begin{equation}
  \label{cancel_x}
  \frac{x[1]_2}{x[1]_1}, \,
  \frac{x[1]_3}{x[1]_4} \to 0 .
\end{equation}
This  condition for  initial cluster variables
denotes that
a cancellation of
the two additional balls
occurs
by connecting to the tubular neighbor of knot $K$ as
explained  in~\cite{Weeks05}
(see also~\cite{ChoMurYok09a}).

\begin{prop}
  \label{prop:cancel}
  In the setting of Thm~\ref{thm:main2},
  when we set an initial cluster $x$-variable as~\eqref{cancel_x},
  % \begin{equation}
  %   \label{cancel_x}
  %   \frac{x[1]_2}{x[1]_1}, \,
  %   \frac{x[1]_3}{x[1]_4} \to 0 ,
  % \end{equation}
  we get 
  a canonical triangulation of $S^3\setminus K$.
\end{prop}

We should note that
such cancellation can occur under
other choices of initial cluster variables.
% there exist several choices of initial cluster variables for
% cancellation of two additional points.

% A triangulation by our octahedra is for the complement of
% knot with two additional points.
% To  endow a hyperbolic structure,
% we need to connect these two additional balls to the tubular
% neighbor of knot $K$ as
% described in~\cite{Weeks05}
% (see also~\cite{ChoMurYok09a}).
% This can be realized when we choose a suitable initial cluster
% variable.

\begin{proof}[Proof of Prop.~\ref{prop:cancel}]
  We need to connect two balls at $v_2$ and $v_3$ (see
  Fig.~\ref{fig:octahedron}) to the tubular neighbor of knot $K$ to 
  get a triangulation of $S^3 \setminus K$.
  For this purpose, we introduce  a triangular  pillow  with a
  pre-drilled tube 
  as in   Fig.~\ref{fig:pillow_for_N}.
  The pillow is constructed from two hyperbolic tetrahedra as in
  Fig.~\ref{fig:pillow_for_N},
  and
  we see 
  that there exists a drilled tube connecting two vertices
  (see~\cite{Weeks05}). 
  By use of other hyperbolic tetrahedra whose vertex orderings are
  opposite to those in Fig.~\ref{fig:pillow_for_N},
  we have another type of a triangular  pillow as in
  Fig.~\ref{fig:pillow_for_W}.

  In both Figs.~\ref{fig:pillow_for_N}
  and~\ref{fig:pillow_for_W},
  we assign edge parameters $c_a$  for each edge.
  Shape parameters of tetrahedra are given from~\eqref{Zickert_c},
  and we get
  \begin{equation}
    \label{c_pillow}
    \frac{c_3}{c_2}=0 .
  \end{equation}
  Because of their opposite vertex orderings,
  a sum of the extended Rogers dilogarithm functions~\eqref{extended_Rogers}
  for two pillows vanishes.

  We insert and 
  glue the pillow in Fig.~\ref{fig:pillow_for_N}
  (resp. Fig.~\ref{fig:pillow_for_W})
  to  the triangular
  surface $x_2 x_3 x_4$  in $\triangle_N$
  (resp.  $x_1 x_2 x_3$ in $\triangle_W$)
  in
  the octahedron assigned to the first crossing~$\mathsf{R}$.
  Pre-drilled tubes of the pillows
  connect both vertices $v_2$ and $v_3$ to $v_1$
  in  Fig.~\ref{fig:octahedron}.
  To conclude,
  we obtain a
  valid triangulation of $S^3 \setminus K$.
  As identical edges have same edge parameters,
  we find that
  a condition~\eqref{c_pillow} gives
  $\frac{x[1]_2}{x[1]_1} = \frac{x[1]_3}{x[1]_4}=0$.
\end{proof}

\begin{figure}[thbp]
  \centering
  \begin{minipage}{0.45\linewidth}
    \centering
    \includegraphics{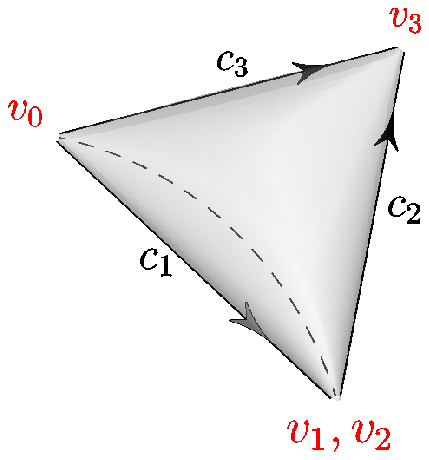}
  \end{minipage}
  \hfill
  \begin{minipage}{0.45\linewidth}
    \centering
    \includegraphics{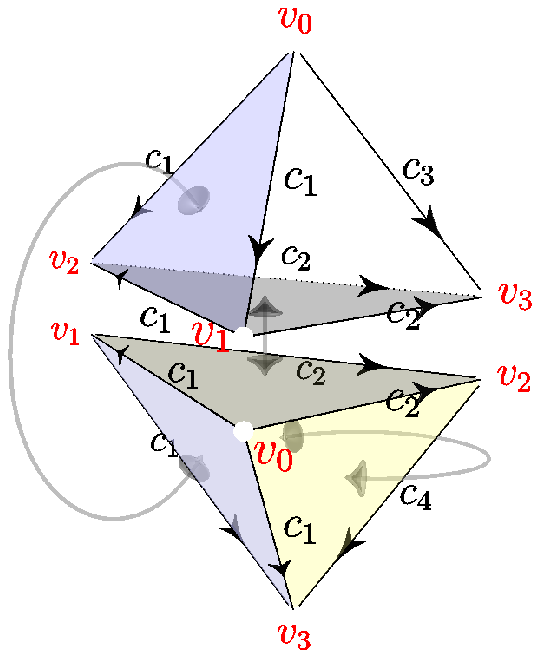}
  \end{minipage}
  \caption{A pillow with a pre-drilled tube (left) is constructed
    from two ideal tetrahedra (right) by gluing colored faces together.
    A dashed curve denotes a  tube connecting two vertices.
    Here $c_a$ is a edge parameter.
  }
  \label{fig:pillow_for_N}
\end{figure}

\begin{figure}[hbtp]
  \centering
  \begin{minipage}{0.45\linewidth}
    \centering
    \includegraphics{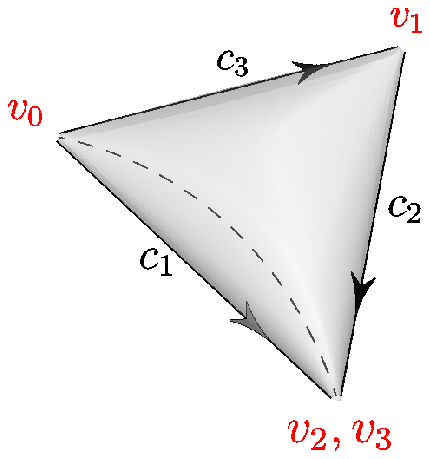}
  \end{minipage}
  \hfill
  \begin{minipage}{0.45\linewidth}
    \centering
    \includegraphics{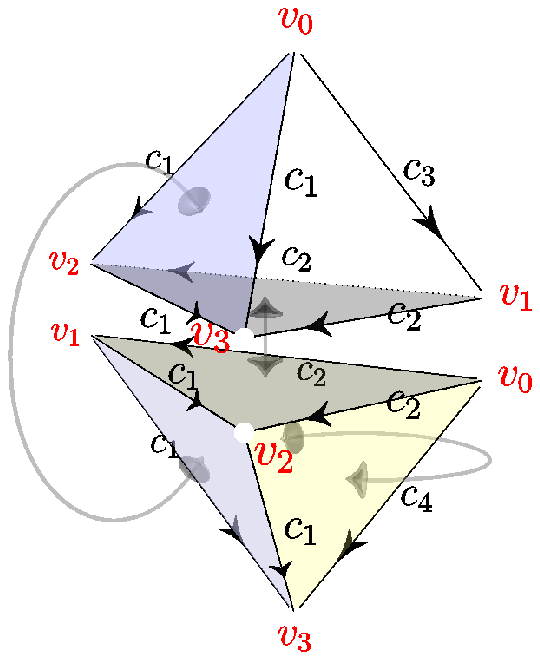}
  \end{minipage}
  \caption{Another pillow with a pre-drilled tube (left) is given from
    two hyperbolic tetrahedra (right).
  }
  \label{fig:pillow_for_W}
\end{figure}

%%%%%%%%%
%\section{Concluding Remarks}
%%%%%%%%%
\section*{Acknowledgments}
The authors would like to thank Jun~Murakami for stimulating
discussions and for comments on the manuscript.
Thanks are also to Rinat Kashaev for bringing~\cite{Dynnik02a} to our
attention.
The work of KH is supported in part by JSPS KAKENHI Grant Number~23340115,
24654041.
%, 22540069.
The work of RI is partially supported by JSPS KAKENHI Grant Number~22740111.

%%%%%%%%%%%%%% newpage
%\bibliographystyle{halphaKH}
%\bibliography{_def,gravity,gravity2,square,math,ba,tba,math5,vm,square2,math4,qalg,math3,math2,poisson,geometry,soliton,cft,knot,tqft,comb,number}
%%%%%%%%%%%%%%%%%%

\end{document}